 \documentclass[a4paper,13pt,abstracton]{article}
 \usepackage[utf8]{inputenc}
   \providecommand{\keywords}[1]{\textbf{\textit{Key words:}} #1}
   \usepackage{amsmath,amsthm,verbatim,amssymb,amsfonts,amscd, graphicx,mathrsfs}
 \usepackage{geometry}
 \usepackage{tabularx}
 \usepackage{diagbox}
 \numberwithin{equation}{section}
 \newtheorem{thm}{Theorem}[section]
 \newtheorem{lem}[thm]{Lemma}
 \newtheorem{define}[thm]{Definition}
 \newtheorem{cor}[thm]{Corollary}
 \newtheorem{prop}[thm]{Proposition}
 \newtheorem{rmk}[thm]{Remark}

 \newtheorem{ex}[thm]{Example}
  
 \usepackage{enumerate}
 \usepackage[square,comma,super,numbers,sort&compress]{natbib}
 \usepackage{bm}
  \usepackage{hyperref}
 \hypersetup{colorlinks=true, urlcolor=Cerulean, citecolor=red}
 \usepackage[T1]{fontenc}
 \usepackage{nameref,cleveref}
 \usepackage{nicefrac,xfrac}
 \usepackage[all]{xy}
 \usepackage{color}
\usepackage[section]{placeins}
\usepackage{indentfirst}

\begin{document}
\title{\textbf {Homogeneous ACM bundles on isotropic Grassmannians}}

\author{Rong Du \thanks{School of Mathematical Sciences
Shanghai Key Laboratory of PMMP,
East China Normal University,
Rm. 312, Math. Bldg, No. 500, Dongchuan Road,
Shanghai, 200241, P. R. China,
rdu@math.ecnu.edu.cn.
},
Xinyi Fang
\thanks{School of Mathematical Sciences
Shanghai Key Laboratory of PMMP,
East China Normal University,
No. 500, Dongchuan Road,
Shanghai, 200241, P. R. China,
2315885681@qq.com.
}
and Peng Ren
\thanks{School of Mathematical Sciences
Shanghai Key Laboratory of PMMP,
East China Normal University,
No. 500, Dongchuan Road,
Shanghai, 200241, P. R. China,
ren194@126.com.
All of the authors are sponsored by Innovation Action Plan (Basic research projects) of Science and Technology Commission of Shanghai Municipality (Grant No. 21JC1401900) and Science and Technology Commission of Shanghai Municipality (Grant No. 18dz2271000).
}
}

\date{}
\maketitle


\begin{abstract}
In this paper, we characterize homogeneous arithmetically Cohen-Macaulay (ACM) bundles over isotropic Grassmannians of types $B$, $C$ and $D$ in term of step matrices. We show that there are only finitely many irreducible homogeneous ACM bundles by twisting line bundles over these isotropic Grassmannians. So we classify all homogeneous ACM bundles over isotropic Grassmannians combining the results on usual Grassmannians by Costa and Mir{\'o}-Roig. Moreover, if the irreducible initialized homogeneous ACM bundles correspond to some special highest weights, then they can be characterized by succinct forms.
\end{abstract}

\section{Introduction}
Vector bundles over a projective variety $X$ are fundamental research objects in algebraic geometry.  However, little is known about vector bundles over general algebraic varieties. Many particular classes of vector bundles have been studied in recent years. One of them is called arithmetically Cohen-Macaulay (ACM) bundles which are defined by the vanishments of all intermediate cohomology groups $H^i(X, E(t))$ for $0<i<\text{dim} X$ and all $t\in \mathbb{Z}$. Such bundles correspond to maximal Cohen-Macaulay modules over the associated graded ring. These modules reflect lots properties of the corresponding ring, so ACM bundles reflect relevant properties of the algebraic variety.

It is well-known that Horrocks showed that a vector bundle without intermediate cohomology on the projective space splits (\cite{horrocks1964vector}\cite{okonek1980vector}). Since this result was established, the study of the indecomposable ACM bundles on a given variety has been drawn the attentions by many mathematicians. The problem of classifying ACM bundles has been taken up only in some special cases. There are many papers on ACM bundles over surfaces since two is the lowest non-trivial dimension of the varieties for ACM bundles (for example, see \cite{Ballico2021}, \cite{Casanellas2011}, \cite{Notari2017}, \cite{Faenzi2008}, \cite{Watanabe2008}, \cite{Yoshioka2021}). There has also been work on ACM bundles on particular higher dimensional varieties such as Fano $3$-folds (\cite{Brambilla2011}, \cite{Casnati2015}),  Calabi-Yau $3$-folds  (\cite{Filip2014}) and hypesurfaces (\cite{Ravindra2019}).  Recently, Costa and Mir{\'o}-Roig used the Bott-Borel-Weil theorem to classify the irreducible homogeneous ACM bundles on Grassmannians (\cite{costa2016homogeneous}), i.e. isotropic Grassmannians of type $A$.  The aim of this paper is to classify all irreducible homogeneous ACM bundles on isotropic Grassmannians $X=G/P(\alpha_k)$ of types $B$, $C$ and $D$.  Therefore we finish classifying all homogeneous ACM bundles over isotropic Grassmannians. We still use step matrices  combining the Borel-Bott-Weil theorem to characterize them, but the situation of these ACM bundles is more complicated than that of Grassmannians of type A. So we separate the step matrix of the irreducible homogeneous vector bundle $E$ into three parts according to the Killing forms of $\lambda+\rho-t\lambda_k$ with positive roots, where $\lambda$ is the highest weight of $E$, $\rho$ is the sum of all fundamental weights and $\lambda_k$ is the $k$-th fundamental weight. Furthermore, since the position of the $n$-th simple root of the Lie algebra of $G$ is not quite similar to other roots, we need to consider the case $k=n$ alone.
However, we can unify our results and have the main theorem finally.

\begin{thm}
Let $E_\lambda$ be an initialized irreducible homogeneous vector bundle with highest weight $\lambda$ over $G/P(\alpha_k)$ of type $B$, $C$ or $D$. Let $T_{k,\lambda}=(t_{ij})$ be its step matrix.
Denote $n_l:=\#\{t_{ij}|t_{ij}=l \}$. Then $E_\lambda$ is an ACM bundle if and only if $n_l\geq1$ for any integer $l\in [1,M_{k,\lambda}],$
where $M_{k,\lambda}= max\{t_{ij}\}$.
\end{thm}

From the main theorem, we can get the following corollary.

\begin{cor}
There are only finitely many irreducible homogeneous ACM bundles up to tensoring a line bundle over $G/P(\alpha_k)$ of types $B$, $C$ and $D$. In particular, the moduli space of projective bundles produced by irreducible homogeneous ACM bundles consists of finite points.
\end{cor}
\paragraph{Plan of the paper} In Section 2, we introduce some theorems on rational homogeneous spaces, especially on the theory of irreducible homogeneous vector bundles. In Section 3, we show our main theorem on classifying the initialized irreducible homogeneous ACM bundles on isotropic Grassmannians of types $B$, $C$ and $D$. Especially, we show that there are only finitely many irreducible homogeneous ACM bundles by twisting line bundles over these isotropic Grassmannians. Moreover, if the irreducible initialized homogeneous ACM bundles correspond to special highest weights, then we present some simple criteria to characterize them.

 \paragraph{Notation and convention}

\begin{itemize}
\item $B_n$: the simple Lie group with Dynkin diagram $B_n$;
\item  $C_n$: the simple Lie group with Dynkin diagram $C_n$;
\item  $D_n$: the simple Lie group with Dynkin diagram $D_n$;
\item  $e_i$: orthonormal basis of the $\mathbb{R}$-vector space spanned by the vectors corresponding to the simple roots;
\item  $(a_1,\dots,a_n)$: $a_1e_1+\dots+a_ne_n$;
\item  $\Phi^+$: the set of positive roots;
\item  $\Phi^-$: the set of negative roots;
\item  $\lambda_k$: the $k$-th fundamental weight;
\item  $(\cdot,\cdot)$: the Killing form;
\item $E_\lambda$: the irreducible homogeneous vector bundle with highest weight $\lambda$;
\item $G/P(\alpha_k)$: the isotropic Grassmannian with semisimple complex Lie group $G$ and parabolic subgroup $P(\alpha_k)$;
\item $T_{k, \lambda}^{Z}$: the step matrix of $E_\lambda$ on the isotropic Grassmannian $G/P(\alpha_k)$, where G is of type $Z$.
\end{itemize}

\section{Preliminaries}
Throughout this paper, all algebraic varieties and morphisms will be defined over the field $\mathbb{C}$.

\subsection{Weights}
Let $G$ be a semisimple complex Lie group and $H$ be a fixed maximal torus of $G$. Denote their Lie algebras by
$\mathfrak{g}$ and $\mathfrak{h}$ respectively. Let $\Phi$ be its root system and $\Delta=\{\alpha_1,...,\alpha_n\}\subset\Phi$ be a set of fixed simple roots.

The \emph{weight lattice} $\Lambda$ of $G$ consists of the linear function $\lambda: \mathfrak{h}\to \mathbb{C}$ such that $\frac{2(\lambda,\alpha)}{(\alpha,\alpha)}\in\mathbb{Z}$ for all $\alpha\in\Phi$.   An element in $\Lambda$ is called a weight. A weight $\lambda\in\Lambda$ is said to be \emph{dominant} if $\frac{2(\lambda,\alpha)}{(\alpha,\alpha)}$ are non-negative for $\alpha\in\Delta$ and \emph{strongly dominant} if these integers are positive. We call $\lambda_i$ the \emph{fundamental dominant weights} if $\frac{2(\lambda_i,\alpha_j)}{(\alpha_j,\alpha_j)}=\delta_{ij}.$ From the definition, we can easily to see that a weight $\mu=\sum\limits_{u=1}^na_u\lambda_u $ is dominant if $a_u\geq 0$ and strongly dominant if $a_u>0$.

Let $V$ be a representation of $\mathfrak{g}.$ The weight lattice of $\Lambda(V)=\{\lambda\in\Lambda| h.v=\lambda(h)v \text{ for all }h\in\mathfrak{h}\}$. A weight $\lambda\in\Lambda(V)$ is called the \emph{highest weight} of $V$ if $\lambda+\alpha$ is not a weight in $\Lambda(V)$ for any $\alpha \in \Phi^+$.

\subsection{Rational homogeneous spaces}
 Let us introduce some concepts on rational homogeneous spaces.
\begin{define}
A closed subgroup $P$ of $G$ is called \emph{parabolic} if the quotient space $G/P$ is projective.
\end{define}

 Let $I\subset\Delta$ be a subset of simple roots. Define \[\Phi^-(I):=\{\alpha\in\Phi^-|\alpha=\sum\limits_{\alpha_{i}\notin I}p_i\alpha_i\}.\]
Let \[\mathcal{P}(I):=\mathfrak{h}\bigoplus (\oplus_{\alpha\in\Phi^+}\mathfrak{g}_\alpha)\bigoplus (\oplus_{\alpha\in\Phi^-(I)}\mathfrak{g}_\alpha)\]
and $P(I)$ be the subgroup of $G$ such that the Lie algebra of $P(I)$ is $\mathcal{P}(I).$ We have the following theorem to describe all parabolic subgroups of $G$.

\begin{thm}(see \cite{ottaviani1995rational} Theorem 7.8)
Let $G$ be a semisimple simply connected Lie group and $P$ be a parabolic subgroup of G. Then There exists $g\in G$ and $I\subset \Delta$ such that \[g^{-1}Pg=P(I).\]
\end{thm}
From this classification theorem, we always use $P(I)$ to denote the parabolic subgroup of $G$.

\begin{define}
A \emph{rational homogeneous space} $X$ is a variety with the form \[G/P\simeq G_1/P(I_1)\times G_2/P(I_2)\times \cdots \times G_m/P(I_m),\]
where every $P(I_i)$ is a parabolic subgroup of the simple Lie group $G_i$. Every rational homogeneous space $G_i/P(I_i)$ is called the generalized flag manifold.
\end{define}

  In this paper, we focus on the \emph{isotropic Grassmannian} $G/P(I)$ which is a generalized flag manifold with $\#|I|=1.$
\subsection{Homogeneous vector bundles}
Now we want to introduce an important class of vector bundles on the rational homogeneous space $G/P$.

\begin{define}
Over $G/P$, a vector bundle $E$ is called \emph{homogeneous} if there exists an action $G$ over $E$ such that the following diagram commutes

\centerline{
    \xymatrix{   G\times E \ar[r]\ar[d]& E\ar[d]\\
    G\times G/P \ar[r] & G/P. }}
\end{define}
\begin{rmk}
\emph{\begin{enumerate}
\item[1.] A homogeneous vector bundle over $G/P$ can be represented by $G\times_PV_\nu$, where $\nu:P\to V_\nu$ is a represetation of $P.$
\item[2.] If a representation $\nu: P\to V_\nu$ is irreducible, then we call $E$ an \emph{irreducible homogeneous vector bundle}.
\end{enumerate}}
\end{rmk}

Generally, homogeneous vector bundles over $G/P$ can be classified by the filtration of the irreducible homogeneous vector bundles. Hence we only consider the irreducible homogeneous vector bundles. We first introduce the classification of the irreducible representations of parabolic subgroups.

 \begin{prop}(See \cite[Proposition 10.9]{ottaviani1995rational})
Let $I=\{\alpha_1,\dots,\alpha_k\}$ be a subset of simple roots. Let $\lambda_1,\dots,\lambda_k$ be the corresponding fundamental weights. Then all the irreducible representations of $P(I)$ are $$V\otimes L^{n_1}_{\lambda_1}\otimes\dots\otimes L^{n_k}_{\lambda_k},$$
where $V$ is a representation of $S_P$ (the semisimple part of $P$), $n_i\in\mathbb{Z}$ and $L_{\lambda_i}$ is a one-dimensional representation with weight $\lambda_i$.
\end{prop}

Notice that the weight lattice of $S_P$ can be embedded in the weight lattice of $G$. If $\lambda$ is the highest weight of an irreducible representation $V$ of $S_P$, then $\lambda+\sum\limits_{u=1}^kn_u\lambda_u$ is the highest weight of an irreducible representation of $V\otimes L^{n_1}_{\lambda_1}\otimes\dots\otimes L^{n_k}_{\lambda_k}$.

\begin{rmk}\label{dominant}
\emph{\begin{enumerate}
\item[1. ]In this paper, we denote $E_\lambda$ by the homogeneous bundle arising from the irreducible representation of $P$ with highest weight $\lambda.$
\item [2.] The irreducible representation of semisimple Lie group is determined by its highest weight.
Hence if  $E_\lambda$ is an irreducible homogeneous vector bundle over $G/P(I)$ with $\lambda=\sum a_u\lambda_i$, then $a_u\geq0$ for $\alpha_u\notin I.$
\end{enumerate}}
\end{rmk}

\subsection{Borel-Bott-Weil Theorem}
The Borel-Bott-Weil theorem is a powerful tool to compute the sheaf cohomology groups of irreducible homogeneous bundles. In order to present this theorem, we firstly introduce the following definition which can be found in Chapter 11 of Ottaviani's nice survey paper\cite{ottaviani1995rational}.
\begin{define}
Let $\lambda$ be a weight of a representation.
\begin{enumerate}
\item[1)]$\lambda$ is called \emph{singular} if there is $\alpha\in \Phi^+ $ such that $(\alpha,\lambda)=0.$
\item[2)] $\lambda$ is called \emph{regular of index $p$} if it is not singular and if there exactly $p$ roots $\alpha_1,\dots,\alpha_p\in\Phi^+ $ such that $(\lambda,\alpha_i)<0.$
\end{enumerate}
\end{define}

Now we can introduce the Borel-Bott-Weil theorem.

\begin{thm}[Borel-Bott-Weil, see \cite{ottaviani1995rational}]\label{borel bott weil} Let $E_\lambda$ be an irreducible homogeneous vector bundle over $G/P.$
\begin{enumerate}
\item[1)] If $\lambda+\rho$ is singular, then $$H^i(G/P,E_\lambda)=0, \forall i\in\mathbb{Z}.$$
\item[2)] If $\lambda+\rho$ is regular of index p, then $$H^i(G/P,E_\lambda)=0, \forall i\neq p,$$
 and $$H^p(G/P,E_\lambda)=G_{w(\lambda+\rho)-\rho},$$ where $\rho=\sum\limits_{i=1}^n\lambda_i$ and $w(\lambda+\rho)$ is the unique element of the fundamental Weyl chamber of G which is congruent to $\lambda+\rho$ under the action of the Weyl group.
 \end{enumerate}
\end{thm}

\section{Classification of irreducible homogeneous ACM bundles on isotropic Grassmannians}
\subsection{ACM bundles}
We first introduce ACM bundles on a projective algebraic variety.

\begin{define}
Let $X\subseteq\mathbb{P}^N$ be a projective variety and $\mathcal{O}_X(1):=\mathcal{O}_{\mathbb{P}^N}(1)|_X$. A vector bundle $E$ over $X$ is called \emph{arithmetically Cohen Macauley} (\emph{ACM} for short) if $$H^i(X,E(t))=0, ~\text{where}~E(t)=E\otimes\mathcal{O}_X(t), ~ \text{for all}~i=1,2,\dots,\dim X-1 ~\text{and}~ t \in \mathbb{Z}.$$
\end{define}

Generally, it is hard to classify all ACM bundles on a variety. In this paper, we shall classify the irreducible homogeneous ACM bundles on isotropic Grassmannians of types $B$, $C$ and $D$. It is easy to see that $E$ is an ACM bundle if and only if $E(t)$ is an ACM bundle.
So for simplicity, we introduce the following definition.

\begin{define}
Given a projective variety $(X,\mathcal{O}_X(1))$, a vector bundle ${E}$ on X is called \emph{initialized} if $$H^0(X, {E}(-1))=0$$ and $$H^0(X, {E})\neq0.$$
\end{define}

For an irreducible homogeneous vector bundle over an isotropic Grassmannian, we have the following lemma.
 \begin{lem}\label{initial} Let $X=G/P(\alpha_k)$ be an isotropic Grassmannian.
If $E_\lambda$ is initialized with $\lambda=a_1\lambda_1+\dots+a_n\lambda_n$, then $a_{k}=0.$
\end{lem}
\begin{proof}
By the Borel-Bott-Weil theorem (Theorem \ref{borel bott weil}), $H^0(X,E_\lambda)\neq0$ is equivalent to $\lambda+\rho$ being regular of index 0, which means that $\lambda+\rho$ is strongly dominant. So $a_i+1>0.$ Meanwhile, $H^0(X,E_\lambda(-1))=0$ shows that $\lambda+\rho-\lambda_k$ is not strongly dominant, which means $a_k+1-1\leq0$. Hence $a_k=0.$
\end{proof}

\subsection{ACM bundles on $G/P(\alpha_k)$ for $k\neq n$}

Let $G$ be a simply connected simple Lie group with the Dynkin diagram of types $A_n$, $B_n$, $C_n$ or $D_n$ as follows.
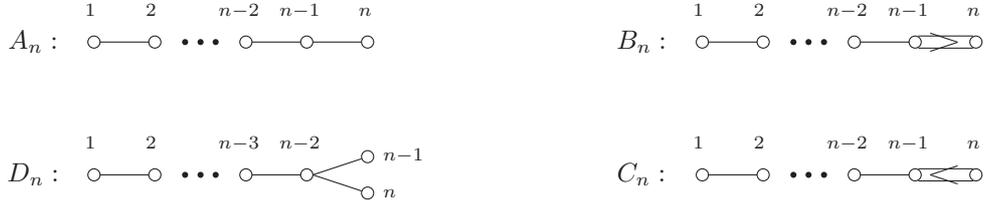
\setlength{\unitlength}{0.4mm}
\begin{center}
\begin{picture}(280,0)(0,120)
\put(10,100){\circle{4}} \put(30,100){\circle{4}}
\put(60,100){\circle{4}} \put(80,100){\circle{4}}
\put(12,100){\line(1,0){16}} \put(40,100){\circle*{2}}
\put(45,100){\circle*{2}} \put(50,100){\circle*{2}}
 \put(62,100){\line(1,0){16}}
 \put(100,100){\circle{4}}
 \put(82,100){\line(1,0){16}}
\put(-10,100){\makebox(0,0)[cc]{$A_n:$}}
\put(7,110){$_1$}
\put(27,110){$_2$}
\put(51,110){$_{n-2}$}
\put(71,110){$_{n-1}$}
\put(97,110){$_{n}$}

\put(210,100){\circle{4}} \put(230,100){\circle{4}}
\put(260,100){\circle{4}} \put(280,100){\circle{4}}
\put(212,100){\line(1,0){16}} \put(240,100){\circle*{2}}
\put(245,100){\circle*{2}} \put(250,100){\circle*{2}}
 \put(262,100){\line(1,0){16}}
 \put(300,100){\circle{4}}
 \put(281,102){\line(1,0){18}}
 \put(281,98){\line(1,0){18}}
 \put(285,103){\line(3,-1){9}}
 \put(285,97){\line(3,1){9}}
\put(190,100){\makebox(0,0)[cc]{$B_n:$}}
\put(207,110){$_1$}
\put(227,110){$_2$}
\put(251,110){$_{n-2}$}
\put(271,110){$_{n-1}$}
\put(297,110){$_{n}$}
 \end{picture}
\end{center}
\vspace{.3cm}

\begin{center}
\begin{picture}(280,20)(0,120)
\put(10,100){\circle{4}} \put(12,100){\line(1,0){16}}
\put(30,100){\circle{4}} \put(40,100){\circle*{2}}
\put(45,100){\circle*{2}} \put(50,100){\circle*{2}}
 \put(60,100){\circle{4}}
\put(62,100){\line(1,0){16}} \put(80,100){\circle{4}}
\put(82,100){\line(3,1){16}} \put(100,94){\circle{4}} \put(82,100){\line(3,-1){16}} \put(100,106){\circle{4}}
\put(-10,100){\makebox(0,0)[cc]{$D_n:$}}
\put(7,110){$_1$}
\put(27,110){$_2$}
\put(51,110){$_{n-3}$}
\put(71,110){$_{n-2}$}
\put(105,106){$_{n-1}$}
\put(105,94){$_{n}$}

\put(210,100){\circle{4}} \put(230,100){\circle{4}}
\put(260,100){\circle{4}} \put(280,100){\circle{4}}
\put(212,100){\line(1,0){16}} \put(240,100){\circle*{2}}
\put(245,100){\circle*{2}} \put(250,100){\circle*{2}}
 \put(262,100){\line(1,0){16}}
 \put(300,100){\circle{4}}
 \put(281,102){\line(1,0){18}}
 \put(281,98){\line(1,0){18}}
 \put(285,100){\line(3,-1){9}}
 \put(285,100){\line(3,1){9}}
\put(190,100){\makebox(0,0)[cc]{$C_n:$}}
\put(207,110){$_1$}
\put(227,110){$_2$}
\put(251,110){$_{n-2}$}
\put(271,110){$_{n-1}$}
\put(297,110){$_{n}$}
 \end{picture}
\end{center}
\vspace{1cm}

One can see that the positions of the $n$-th points in the graphs of types $B_n$, $C_n$ and $D_n$ are obviously different from that of type $A_n$, which causes the complexity of isotropic Grassmannians of types $B_n$, $C_n$ and $D_n$. In this section, we shall classify irreducible homogeneous ACM bundles over isotropic Grassmannians $G/P(\alpha_k)$ of types $B$, $C$ and $D$.
 We focus on $k<n$ for types $B$ and $C$, and on $k<n-1$ for type $D$. The reason that we do not consider $k=n-1$ for type $D$ is  $D_n/P(\alpha_{n-1})\simeq D_n/P(\alpha_{n})$.

We give the explicit forms of the positive roots and fundamental weights of Lie groups of types $B$, $C$ and $D$.

\begin{lem} \label{fund} (\cite{kuznetsov2016exceptional} Section 9)
We define \[e=\left\{\begin{matrix}\frac{1}{2}& \text{if G is of type B},\\
1& \text{if G is of type C},\\
0& \text{if G is of type D}.\\
\end{matrix}\right.\]
For Lie algebras of types $B$, $C$ and $D$, we can take orthogonal bases of the $\mathbb{R}$-vector space spanned by the vectors corresponding to the simple roots of these Lie algebras such that the positive roots are $$\Phi^+_B =\{e_i+e_j\}_{i<j}\cup\{e_i-e_j\}_{i<j}\cup\{e_i\}_i,$$
$$\Phi^+_C =\{e_i+e_j\}_{i\leq j}\cup\{e_i-e_j\}_{i<j},$$
$$\Phi^+_D =\{e_i+e_j\}_{i<j}\cup\{e_i-e_j\}_{i<j}.$$
The fundamental weights  $$\lambda_i^{B,C,D}=e_1+\dots+e_i,$$ for $i\leq n-2+2e$ and
$$\lambda_n^{B,C}=e(e_1+\dots+e_{n-1}+e_n),$$
$$\lambda_{n-1}^D=\frac{1}{2}(e_1+\dots+e_{n-1}-e_n),$$
$$\lambda_n^{D}=\frac{1}{2}(e_1+\dots+e_{n-1}+e_n).$$
$$So ~\rho=(n+e-1)e_1+(n+e-2)e_2+\dots+ee_n.$$
\end{lem}
\begin{proof}
Refer to the Appendix of Carter\cite{carter2005lie} or Lecture 15 and 18 of Fulton-Harris\cite{fulton2013representation}.
\end{proof}

We have the following lemma about the dimensions of isotropic Grassmannians, which can be found in \cite{snow1989homogeneous}  Section 9.
\begin{lem}
 The dimension of an isotropic Grassmannian $G/P(\alpha_k)$ is \[\dim G/P(\alpha_k)=\left\{\begin{matrix}\frac{k(4n+1-3k)}{2},& ~\text{if }G=B_n ~\text{and } ~C_n,\\
 \frac{k(4n-1-3k)}{2},& ~\text{if }G=D_n (k\neq n-1).\\
 \end{matrix}\right.\]
\end{lem}
Let us first define the step matrices of irreducible homogeneous vector bundles over $G/P(\alpha_k)$ when $k\neq n$.

\begin{define}
Let $E_\lambda$ be an irreducible homogeneous vector bundle over $X=G/P(\alpha_k)$ with $\lambda=a_1\lambda_1+\dots+a_n\lambda_n.$  We define its \emph{step matrix} $T_{k,\lambda}=(P_{k,\lambda} ,Q_{k,\lambda} ,R_{k,\lambda} )$ which is a $k\times (2n-k)$ matrix.
Here
$P_{k,\lambda}^{B,C,D}=(p_{ij})$ is a $k\times (n-k)$-matrix, where\[p_{ij}=
\sum\limits_{u=1+k-i}^{k+j-1}a_{u}+j+i-1,~ 1\leq i\leq k,~1\leq j\leq n-k,\]
i.e., $P^{B,C,D}_{k,\lambda}=$
  $$  \left(\begin{matrix}
a_k+1&a_k+a_{k+1}+2&\dots& \sum\limits_{u=k}^{n-1}a_u+n-k\\
a_{k-1}+a_k+2&a_{k-1} +a_k+a_{k+1}+3&\dots&\sum\limits_{u=k-1}^{n-1}a_u+n-k+1\\
\vdots&\vdots&\ddots&\vdots\\
\sum\limits_{u=1}^{k}a_u+k&\sum\limits_{u=1}^{k+1}a_u+k+1&\dots&\sum\limits_{u=1}^{n-1}a_u+n-1\\
   \end{matrix}   \right);$$

$Q_{k,\lambda}^{B,C}=(q_{ij})$ is a $k \times (n-k)$-matrix, where
\[q_{ij}=
\sum\limits_{u=k+1-i}^{n-1}a_{u}+\sum\limits_{u=n+1-j}^{n-1}a_{u}+2ea_n+n-k+j+i-2+2e, ~1\leq i\leq k,~1\leq j\leq n-k,\]

i.e., $Q_{k,\lambda}^{B,C}=$ \begin{normalsize}
    $$   \left(\begin{smallmatrix}
\sum\limits_{u=k}^{n-1}a_u+2ea_n+n-k+2e&\sum\limits_{u=k}^{n-1}a_u+a_{n-1}+2ea_n+n-k+1+2e&\dots&\sum\limits_{u=k}^{n-1}a_u+\sum\limits_{u=k+1}^{n-1}a_u+2ea_n+2n-2k-1+2e\\
\sum\limits_{u=k-1}^{n-1}a_u+2ea_n+n-k+1+2e&\sum\limits_{u=k-1}^{n-1}a_u+a_{n-1}+2ea_n+n-k+2+2e&\dots&\sum\limits_{u=k-1}^{n-1}a_u+\sum\limits_{u=k+1}^{n-1}a_u+2ea_n+2n-2k+1+2e\\
\vdots&\vdots&\ddots&\vdots\\
\sum\limits_{u=1}^{n-1}a_u+2ea_n+n-1+2e &\sum\limits_{u=1}^{n-1}a_u+a_{n-1}+2ea_n+n+2e&\dots&\sum\limits_{u=1}^{n-1}a_u+\sum\limits_{u=k+1}^{n-1}a_u+2ea_n+2n-k-2+2e\\
   \end{smallmatrix}    \right);$$
\end{normalsize}

$Q_{k,\lambda}^D=(q_{ij})$ is a $k \times (n-k)$-matrix, where
\[q_{ij}=
\sum\limits_{u=k+1-i}^{n-2}a_{u}+\sum\limits_{u=n+1-j}^{n}a_{u}+n-k+j+i-2,~ 1\leq i\leq k,~1\leq j\leq n-k,\]
i.e., $Q_{k,\lambda}^D=$ \begin{footnotesize}
    $$ \left(\begin{matrix}
\sum\limits_{u=k}^{n-2}a_u+a_n+n-k&\sum\limits_{u=k}^{n}a_u+n-k+1&\dots&\sum\limits_{u=k}^{n}a_u+\sum\limits_{u=k+1}^{n-2}a_u+2n-2k-1\\
\sum\limits_{u=k-1}^{n-2}a_u+a_n+n-k+1&\sum\limits_{u=k-1}^{n}a_u+n-k+2&\dots&\sum\limits_{u=k-1}^{n}a_u+\sum\limits_{u=k+1}^{n-2}a_u +2n-2k\\
\vdots&\vdots&\ddots&\vdots\\
\sum\limits_{u=1}^{n-2}a_u+a_n+n-1 &\sum\limits_{u=1}^{n}a_u+n&\dots&\sum\limits_{u=1}^{n}a_u+\sum\limits_{u=k+1}^{n-2}a_u+2n-k-2\\
   \end{matrix}    \right);$$
   \end{footnotesize}

$R_{k,\lambda}^{B,C}=(r_{ij})$ is a $(k\times k)$-matrix, where
\[r_{ij}=\left\{\begin{matrix}   n-k-1+e+ \frac{1}{2}(
\sum\limits_{u=1+k-i}^{n-1}a_u+
\sum\limits_{v=1+k-j}^{n-1}a_v+2ea_n+j+i),&i\leq j,\\
0,&i>j,\end{matrix}\right.\]
i.e., $R_{k,\lambda}^{B,C}=$
\begin{normalsize}
$$ \frac{1}{2}\left(\begin{smallmatrix}
2\sum\limits_{u=k}^{n-1}a_u+2ea_n+2n-2k+2e&a_{k-1} +2\sum\limits_{u=k }^{n-1}a_u+2ea_n+2n-2k+1+2e&\dots &\sum\limits_{u=1}^{k-1}a_u+2\sum\limits_{u=k}^{n-1}a_u+2ea_n+2n-k-1+2e\\
 0&2\sum\limits_{u=k-1}^{n-1}a_u+2ea_n+2n-2k+2+2e&\dots &\sum\limits_{u=1}^{k-2}a_u+2\sum\limits_{u=k-1}^{n-1}a_u+2ea_n+2n-k+2e\\
 \vdots&\vdots&\ddots&\vdots\\
0&0&\dots&2\sum\limits_{u=1}^{n-1}a_u+2ea_n+2n-2+2e
   \end{smallmatrix}   \right); $$
\end{normalsize}

and
$R_{k,\lambda}^D=(r_{ij})$ is a $(k\times k)$-matrix, where
\[r_{ij}=\left\{\begin{matrix}   n-k-1+ \frac{1}{2}(
\sum\limits_{u=1+k-i}^{n-2}a_u+
\sum\limits_{u=1+k-j}^{n}a_u+j+i),&i< j,\\
0,&i\geq j,\end{matrix}\right.\]
i.e., $R_{k,\lambda}^D=$
\begin{normalsize}
$$\frac{1}{2}\left(\begin{smallmatrix}
0&\sum\limits_{u=k-1}^{n}a_u+\sum\limits_{u=k}^{n-2}a_u+2n-2k+1&\sum\limits_{u=k-2}^{n}a_u+\sum\limits_{u=k}^{n-2}a_u+2n-2k+2&\dots &\sum\limits_{u=1}^{n}a_u+\sum\limits_{u=k}^{n-2}a_u+2n-k-1\\
 0&0&\sum\limits_{u=k-2}^{n}a_u+\sum\limits_{u=k-1}^{n-2}a_u+2n-2k+3&\dots &\sum\limits_{u=1}^{n}a_u+\sum\limits_{u=k-1}^{n-2}a_u+2n-k\\
 \vdots&\vdots&\vdots&\ddots&\vdots\\
0&0&0&\dots&\sum\limits_{u=1}^{n}a_u+\sum\limits_{u=2}^{n-2}a_u+2n-3\\
0&0&0&\dots& 0\\
   \end{smallmatrix}   \right). $$

\end{normalsize}
\end{define}

Now we can claim our main theorem for $G/P(\alpha_k)$ ($k\neq n)$.

\begin{thm}\label{thm<n}
Let $E_\lambda$ be an initialized irreducible homogeneous vector bundle over $X=G/P(\alpha_k)$ ($k\neq n$) with $\lambda=a_1\lambda_1+\dots+a_n\lambda_n$. Let $T_{k,\lambda}=(t_{ij})$ be its step matrix.
Denote $n_l:=\#\{t_{ij} |t_{ij}=l \}$. Then $E_\lambda$ is an ACM vector bundle if and only if $n_l\geq1$ for any integer $l\in [1,M_{k,\lambda} ],$ where \[M_{k,\lambda}=max\{t_{ij}\} =\left\{\begin{matrix}\sum\limits_{u=1}^{n-1}a_u+\sum\limits_{u=k+1}^{n-1}a_u+2ea_n+2n-k-2+2e,& \text{if G is of type B or C},\\
\sum\limits_{u=1}^{n}a_u+\sum\limits_{u=k+1}^{n-2}a_u+2n-k-2,&  \text{if G is of type D}.
\end{matrix}\right.\]
\end{thm}

 \begin{proof}
 If $E_\lambda$ is an ACM bundle then  $$H^i(X,E_{\lambda}(-t))=0,$$
 for all $i=1,\dots, \dim X-1$ and all integers $t$ from the definition.
 The highest weight of $E_{\lambda}(-t)$ is $\lambda-t\lambda_k.$
 By the Borel-Bott-Weil theorem, we know that in order to prove $E_\lambda$ being an ACM bundle is equivalent to showing that $\lambda+\rho-t\lambda_k$ satisfys one of the following conditions for each $t\in \mathbb{Z}$.

\begin{enumerate}
\item[1)] $\lambda+\rho-t\lambda_k$ is regular of index 0;
\item[2)] $\lambda+\rho-t\lambda_k$ is regular of index $\dim X$;
\item[3)] $\lambda+\rho-t\lambda_k$ is singular.
\end{enumerate}

In order to check these conditions, we need to compute the Killing forms of  $\lambda+\rho-t\lambda_k=\sum\limits_{u\neq k}(a_u+1)\lambda_u+(a_k+1-t)\lambda_k$ with positive roots.
For simplicity, we expand $\lambda+\rho-t\lambda_k$ in terms of the choices of the orthonormal bases in Lemma \ref{fund}.

\vspace{.5cm}
$\textbf{Type B,C:}$
$\lambda+\rho-t\lambda_k $ is equal to

\[(\underbrace{\sum\limits_{u=1}^{n-1}a_u +e(a_n+1)+n-1-t,\dots,\sum\limits_{u=k}^{n-1}a_u+e(a_n+1)+n-k-t}_{k~terms},\]
\[\underbrace{\sum\limits^{n-1}_{u=k+1} a_u+e(a_n+1)+n-k-1,\dots,e(a_n+1)}_{n-k~terms}). \]

Firstly, $\lambda+\rho-t\lambda_k$ being regular of index 0 is equivalent to saying that $\lambda+\rho-t\lambda_k$ is strongly dominant, i.e., $a_u+1>0$ $(u\neq k)$ and $1+a_k-t>0.$ By the second statement of Remark \ref{dominant}, we know that $a_u\geq 0~(u\neq k)$. So $\lambda+\rho-t\lambda_k$ being regular of index 0 is equivalent to $1+a_k-t>0.$ By Lemma \ref{initial}, the initialized condition means $a_k=0$. Thus $\lambda+\rho-t\lambda_k$ is regular of index 0 if and only if $t<1$.

If $\lambda+\rho-t\lambda_k$ is regular of index $\dim X=\frac{k(4n+1-3k)}{2}$, then $(\lambda+\rho-t\lambda_k,e_1+e_{k+1})<0$,
which means that \[\sum\limits^{n-1}_{u=k+1} a_u+e(a_n+1)+n-k-1+\sum\limits_{u=1}^{n-1}a_u+ e(a_n+1)+n-1-t<0,\]

i.e.,
\[t>\sum\limits_{u=1}^{n-1}a_u+\sum\limits_{u=k+1}^{n-1}a_u+2ea_n+2n-k-2+2e=M_{k,\lambda}^{B,C}.\]

Conversly, under such a condition, we can get that $$(\lambda+\rho-t\lambda_k,e_p+e_q) $$ for $1\leq p\leq k<q\leq n$ and $1\leq p\leq q\leq k$, and $$(\lambda+\rho-t\lambda_k,e_p-e_q) $$ for $1\leq p\leq k<q\leq n$, are all negative. So the Killing forms of exactly $$(n-k)k+\frac{k(k+1)}{2}+(n-k)k=\frac{k(4n+1-3k)}{2}$$ positive roots  with $\lambda+\rho-t\lambda_k$ are negative.

Thus $\lambda+\rho-t\lambda_k$ is regular of index $\frac{k(4n+1-3k)}{2}$ if and only if $t>M_{k,\lambda}^{B,C}$.

Hence  $E_\lambda$ being an ACM bundle is equivalent to  $\lambda+\rho-t\lambda_k$ being singular for any $t\in[1,M_{k,\lambda}^{B,C}]$ (i.e., there exists a positive root $\alpha$ such that $(\lambda+\rho-t\lambda_k,\alpha)=0$ for any integer $ t\in[1,M_{k,\lambda}^{B,C}]$).

Notice that the sign of $(\lambda+\rho-t\lambda_k,e_p) $ is as same as the sign of $(\lambda+\rho-t\lambda_k,2e_p)$. We replace $e_p$ with $2e_p$ in the following proof so that  we only need to consider positive roots of two forms $(e_p+e_q)_{p\leq q}$ and $(e_p-e_q)_{p<q}$  by Lemma \ref{fund}.

\vspace{.5cm}
$\textbf{Case 1:}$ $(\lambda+\rho-t\lambda_k,e_p-e_q)=0.$ Since the first $k$ elements and the last $n-k$ elements of the vector $\lambda+\rho-t\lambda_k$ are strictly decreasing, we only need to consider the case that $(\lambda+\rho-t\lambda_k,e_{1+k-i}-e_{k+j})=0$ for $1\leq i\leq k$ and $1\leq j\leq n-k$, which is equivalent to
\[\sum\limits^{n-1}_{u=k+j} a_i+e(a_n+1)+n-(k+j)=\sum\limits_{u=1+k-i}^{n-1}a_u+ e(a_n+1)+n-(1+k-i)-t.\]
Then \[t= \sum\limits_{u=1+k-i}^{k+j-1}a_u +i+j-1.\]

$\textbf{Case 2:}$ $(\lambda+\rho-t\lambda_k,e_p+e_q)=0$. Since the last $n-k$ elements of the vector $\lambda+\rho-t\lambda_k$ are positive, we only need to consider the cases that $1\leq p\leq q\leq k$ and $1\leq p\leq k<q\leq n.$

a) $(\lambda+\rho-t\lambda_k,e_{1+k-i}+e_{1+n-j})=0$, where $1\leq i\leq k,~1\leq j\leq n-k$,  is equivalent to
\[\sum\limits_{u=1+k-i}^{n-1}a_u+ e(a_n+1)+n-(1+k-i)-t=-(\sum\limits_{u=1+n-j}^{n-1}a_u+ e(a_n+1)+n -(1+n-j)),\]

i.e., \[t=\sum\limits_{u=k+1-i}^{n-1}a_{u}+\sum\limits_{u=n+1-j}^{n-1}a_{u}+2ea_n+n-k+j+i-2+2e.\]

b) $(\lambda+\rho-t\lambda_k,e_{1+k-i}+e_{1+k-j})=0$, where $1\leq i\leq j\leq k $, is equivalent to
\[\sum\limits^{n-1}_{ u=1+k-i} a_u+e(a_n+1)+n-(1+k-i)-t=-(\sum\limits_{u=1+k-j}^{n-1}a_u+ e(a_n+1)+n-(1+k-j)-t),\]

i.e.,\[t= n-k-1+e+ \frac{1}{2}(
\sum\limits_{u=1+k-i}^{n-1}a_u+
\sum\limits_{v=1+k-j}^{n-1}a_v+2ea_n+j+i).\]

These are the elements in $T_{k,\lambda}^{B,C}.$ Hence  $\lambda+\rho-t\lambda_k$ is singular if and only if $t$ is in $T_{k,\lambda}^{B,C}.$ Hence $E_\lambda$ is an ACM bundle if and only if $n_l\geq 1$ for any integer $l\in[1,M_{k,\lambda}^{B,C}].$

\vspace{.5cm}
$\textbf{Type D}:$ $\lambda+\rho-t\lambda_k $ is equal to \[(\underbrace{\sum\limits_{u=1}^{n-2}a_u+\frac{1}{2}a_{n-1}+\frac{1}{2}a_n+n-1-t,\dots,\sum\limits_{u=k}^{n-2}a_u+\frac{1}{2}a_{n-1}+\frac{1}{2}a_n+n-k-t,}_{k~terms}\]
\[\underbrace{\sum\limits_{u=k+1}^{n-2}a_u+\frac{1}{2}a_{n-1}+\frac{1}{2}a_n+n-k-1,..,\frac{1}{2}a_{n-1}+\frac{1}{2}a_n+1,-\frac{1}{2}a_{n-1}+\frac{1}{2}a_n}_{n-k~terms}).\]
As above arguments, $\lambda+\rho-t\lambda_k$ is regular of index 0 if and only if $t<a_k+1=1$, and $\lambda+\rho-t\lambda_k$ is regular of index $\dim X=\frac{k(4n-1-3k)}{2}$ if and only if  $t>\sum\limits_{u=1}^{n}a_u+\sum\limits_{u=k+1}^{n-2}a_u+2n-k-2=M_{k,\lambda}^D$.
Thus  $E_\lambda$ being an ACM bundles is equivalent to  $\lambda+\rho-t\lambda_k$ being singular for any $t\in[1,M_{k,\lambda}^D]$.

By  Lemma \ref{fund}, there are two types of positive roots $(e_p-e_q)_{p<q}$ and $(e_p+e_q)_{p<q}$.
Then we still have the following two cases.

\vspace{.5cm}
\textbf{Case 1:} $(\lambda+\rho-t\lambda_k,e_{1+k-i}-e_{k+j})=0,$ where $1\leq i\leq k,1\leq j\leq n-k$,  is equivalent to
\begin{small}
\[\left\{\begin{matrix}
\sum\limits_{u=1+k-i}^{n-2}a_u+\frac{1}{2}a_{n-1}+\frac{1}{2}a_n+n-(1+k-i)-t=-\frac{1}{2}a_{n-1}+\frac{1}{2}a_n, &j=n-k,\\
\sum\limits_{u=1+k-i}^{n-2}a_u+\frac{1}{2}a_{n-1}+\frac{1}{2}a_n+n-(1+k-i)-t=\sum\limits_{u=k+j}^{n-2}a_u+\frac{1}{2}a_{n-1}+\frac{1}{2}a_n+n-(k+j), &j<n-k,\\
\end{matrix}\right.\]
\end{small}

i.e., \[t= \sum\limits_{u=1+k-i}^{k+j-1}a_{u}+j+i-1.\]

\textbf{Case 2:}
a) $(\lambda+\rho-t\lambda_k,e_{1+k-i}+e_{1+n-j})=0,$ where $1\leq i\leq k,~1\leq j\leq n-k$, is equivalent to

\begin{footnotesize}

\[\left\{\begin{matrix}
\sum\limits_{u=1+k-i}^{n-2}a_u+\frac{1}{2}a_{n-1}+\frac{1}{2}a_n+n-(1+k-i)-t=-( -\frac{1}{2}a_{n-1}+\frac{1}{2}a_n), &j=1,\\
\sum\limits_{u=1+k-i}^{n-2}a_u+\frac{1}{2}a_{n-1}+\frac{1}{2}a_n+n-(1+k-i)-t=-(\sum\limits_{u=1+n-j}^{n-2}a_u+\frac{1}{2}a_{n-1}+\frac{1}{2}a_n+n-(1+n-j)),&j>1,\\
\end{matrix}\right.\]

\end{footnotesize}
i.e.,  \[t=\sum\limits_{u=k+1-i}^{n-2}a_{u}+\sum\limits_{u=n+1-j}^{n}a_{u}+n-k+j+i-2.\]

b) $(\lambda+\rho-t\lambda_k,e_{1+k-i}+e_{1+k-j})=0,$ where $1\leq i<j\leq k$, is equivalent to

\[\sum\limits_{u=1+k-i}^{n-2}a_u+\frac{1}{2}a_{n-1}+\frac{1}{2}a_n+n-(1+k-i)-t=-(\sum\limits_{u=1+k-j}^{n-2}a_u+\frac{1}{2}a_{n-1}+\frac{1}{2}a_n+n-(1+k-j)-t),\]

i.e.,  \[t= n-k-1+ \frac{1}{2}(
\sum\limits_{u=1+k-i}^{n-2}a_u+
\sum\limits_{u=1+k-j}^{n}a_u+j+i).\]

These are the elements in $T_{k,\lambda}^D$.
Hence $E_\lambda$ is an ACM bundle if and only if $n_l\geq 1$ for any integer $l\in[1,M_{k,\lambda}^D].$

\end{proof}

 \begin{ex}
 Let $E_\mu$ and $E_{\lambda}$ be initialized homogeneous bundles with highest weight $\mu=4\lambda_1+4\lambda_2,\lambda=2\lambda_4+3\lambda_5$ on $OG(3,11)=B_5/P(\alpha_3).$ We can get that
 $T_{3,\mu}^{B}=(P_{3,\mu}^{B},Q_{3,\mu}^{B},R_{3,\mu}^{B})$, where
 \begin{center}
      $P_{3,\mu}^{B}=\left(\begin{matrix}
 1&2\\
 6&7\\
 11&12\\
 \end{matrix}\right),$
  $Q_{3,\mu}^{B}=\left(\begin{matrix}
  3&4\\
  8&9\\
  13&14
  \end{matrix}\right),$
   $R_{3,\mu}^{B}=\left(\begin{matrix}
   \frac{5}{2}&5&\frac{15}{2}\\
   0&\frac{15}{2}&10\\
   0&0&\frac{25}{2}
   \end{matrix}\right),$
 \end{center}
$M_{3,\mu}^{B} =14$ and $n_l\geq1$ for any integer $l\in[1,14]$.

We can also get that
 $T_{3,\lambda}^{B}=(P_{3,\lambda}^{B},Q_{3,\lambda}^{B},R_{3,\mu}^{B})$, where

     \begin{center}
      $P_{3,\lambda}^{B}=\left(\begin{matrix}
 1&4\\
 2&5\\
 3&6\\
 \end{matrix}\right),$
  $Q_{3,\lambda}^{B}=\left(\begin{matrix}
  8&11\\
  9&12\\
  10&13
  \end{matrix}\right),$
   $R_{3,\lambda}^{B}=\left(\begin{matrix}
   6&\frac{13}{2}&7\\
   0&7&\frac{15}{2}\\
   0&0&8
   \end{matrix}\right),$
 \end{center}
$M_{3,\lambda}^{B} =13$ and $n_l\geq1$ for any integer $l\in[1,13]$. Hence they are both ACM bundles.
 \end{ex}
In Corollary \ref{corBlessn}, we will generalize the above example.

 \subsection{ACM bundles on $G/P(\alpha_k)$ for $k=n$}
In this section, we claim our theorem for isotropic Grassmannian $G/P(\alpha_n).$ Similarly, we have the following definition.

  \begin{define} \label{defequaln}
  Let $E_\lambda$ be an irreducible homogeneous vector bundle over $X=G/P(\alpha_n)$ with $\lambda=a_1\lambda_1+\dots+a_n\lambda_n.$  We define its \emph{step matrix} $T_{n,\lambda}=(t_{ij})$ which is a $(n\times n)$-matrix.
For type B and C,
\[t_{ij}=\left\{\begin{matrix}
\frac{1}{2e}(\sum\limits_{u=n+1-i}^{n-1}a_u+
\sum\limits_{v=n+1-j}^{n-1}a_u+2ea_n+j+i-2+2e),&i\leq j,\\
0,&i>j,\end{matrix}\right.\]
i.e., $T_{n,\lambda}^{B,C}=$
   $$\frac{1}{2e}\left(\begin{matrix}
  2ea_n+2e & a_{n-1}+ 2ea_n+1+2e&\dots & \sum\limits_{u=1}^{n-1}a_u+2ea_n+n-1+2e \\
 0& 2a_{n-1} +2ea_n+2+2e&\dots & \sum\limits_{u=1}^{n-1}a_u+a_{n-1}+2ea_n+n+2e \\
\vdots &\ddots&\ddots&\vdots\\
0&\cdots&0& 2\sum\limits_{u=1}^{n-1}a_u+2ea_n+2n-2+2e
   \end{matrix}   \right). $$

For type D,
   \[t_{ij}=\left\{\begin{matrix}
\sum\limits_{u=n-i+1}^{n}a_u+
\sum\limits_{v=n-j+1}^{n-2}a_u+j+i-2,&i< j,\\
0,&i\geq j,\end{matrix}\right.\]
i.e., $T_{n,\lambda}^{D}=$
   $$ \left(\begin{matrix}
0&a_{n}+1&a_n+a_{n-2}+2 & \dots & \sum\limits_{u=1}^{n-2}a_u+a_n+n-1 \\
0&0&\sum\limits_{u=n-2}^{n}a_u+3 &  \dots & \sum\limits_{u=1}^{n}a_u+n \\
  \vdots&\vdots&\vdots&\ddots&\vdots\\
 0&0&0&\dots& \sum\limits_{u=1}^{n}a_u+\sum\limits_{u=2}^{n-2}a_u+2n-3\\
 0&0&0&\dots& 0\\
   \end{matrix}   \right). $$

\end{define}

Now we state our main theorem for $G/P(\alpha_n)$.
\begin{thm}\label{thm=n}
Let $E_\lambda$ be an initialized irreducible homogeneous vector bundle over $X=G/P(\alpha_n)$ with $\lambda=a_1\lambda_1+\dots+a_n\lambda_n$. Let $T_{n,\lambda}=(t_{ij}) $ be its associated matrix.
Denote $n_l:=\#\{t_{ij}  |t_{ij}=l \}$. Then $E$ is an ACM vector bundle if and only if $n_l\geq1$ for any integer $l\in [1,M_{n,\lambda}]$ where \[M_{n,\lambda} =max\{t_{ij}\}=\left\{\begin{matrix}\frac{1}{2e}( 2\sum\limits_{u=1}^{n-1}a_u+2ea_n+2n-2+2e) ,&\text{if G is of type B or C},\\
\sum\limits_{u=1}^{n}a_u+\sum\limits_{u=2}^{n-2}a_u+2n-3,&   \text{if G is of type D}.
\end{matrix}\right.\]

\end{thm}
\begin{proof}
$\textbf{Type B,C:}$  $\lambda+\rho-t\lambda_n$ is equal to
\[(\underbrace{\sum\limits_{u=1}^{n-1}a_u+ e(a_n+1)+n-1-et,\sum\limits_{u=2}^{n-1}a_u+e(a_n+1)+n-2-et,\dots,e(a_n+1)-et}_{n~terms}).\]
As the arguments in Theorem \ref{thm<n}, we still have $t< 1$ if and only if $\lambda+\rho-t\lambda_k$ is regular of index 0,
and $\lambda+\rho-t\lambda_k$ is regular of index $\dim X=\frac{n(n+1)}{2}$ if and only if $t>M_{n,\lambda}^{B,C}$.

Hence $E_\lambda$ being an ACM bundle is equivalent to $\lambda+\rho-t\lambda_k$ being singular for any $t\in[1,M_{n,\lambda}^{B,C}]$. Since the elements of the vector of $\lambda+\rho-t\lambda_n$ are strictly decreasing, it is enough to consider the Killing forms of $\lambda+\rho-t\lambda_k$ with $e_p+e_q$ for $p\leq q$. It is sufficient to consider $(\lambda+\rho-t\lambda_k,e_{n+1-i}+e_{n+1-j})=0$ for $1\leq i\leq j\leq n$, i.e.,
\[\sum\limits^{n-1}_{ u=n+1-i} a_u+e(a_n+1)+n-(n+1-i)-et=-(\sum\limits_{u=n+1-j}^{n-1}a_u+e(a_n+1)+n-(n+1-j) -et).\]
Finally, we get
 \[t=\frac{1}{2e}(\sum\limits_{u=n+1-i}^{n-1}a_u+
\sum\limits_{v=n+1-j}^{n-1}a_v+2ea_n+j+i-2+2e).\]

These are the elements in $T_{n,\lambda}^{B,C}$. As the proof of Theorem \ref{thm<n}, $E_\lambda$ is an ACM bundle if and only if $n_l\geq 1$ for any integer $l\in[1,M_{n,\lambda}^{B,C}].$

\vspace{.5cm}
$\textbf{Type D:}$  $\lambda+\rho-t\lambda_{n}$ is equal to
\[ (\underbrace{\sum\limits_{u=1}^{n-2}a_u+\frac{1}{2}a_{n-1}+\frac{1}{2}a_n+n-1-\frac{1}{2}t,\dots,\frac{1}{2}a_{n-1}+\frac{1}{2}a_n+1-\frac{1}{2}t,-\frac{1}{2}a_{n-1}+\frac{1}{2}a_n-\frac{1}{2}t}_{n~terms}).\]

Similarly, $\lambda+\rho-t\lambda_k$ is regular of index 0 if and only if $t< 1$, and $\lambda+\rho-t\lambda_k$ is regular of index $\dim X=\frac{n(n-1)}{2}$ if and only if  $t>\sum\limits_{u=1}^{n}a_u+\sum\limits_{u=2}^{n-2}a_u+2n-3=M_{n,\lambda}^D$.

Hence $E_\lambda$ being an ACM bundle is equivalent to $\lambda+\rho-t\lambda_n$ being singular for any integer $t\in[1,M_{n,\lambda}^D]$.

As above, we just consider $(\lambda+\rho-t\lambda_n,e_{n-i+1 }+e_{n-j+1})=0,$ with $1\leq i<j\leq n$. That is to say

\begin{large}
\[\left\{\begin{smallmatrix}
\sum\limits_{u=1+n-j}^{n-2}a_u+\frac{1}{2}a_{n-1}+\frac{1}{2}a_n+n-(1+n-j)-\frac{1}{2}t=-( -\frac{1}{2}a_{n-1}+\frac{1}{2}a_n-\frac{1}{2}t),&i=1,\\
 \sum\limits_{u=n-i+1}^{n-2}a_u+\frac{1}{2}a_{n-1}+\frac{1}{2}a_n+n-(n-i+1)-\frac{1}{2}t=-(\sum\limits_{u=n-j+1}^{n-2}a_u+\frac{1}{2}a_{n-1}+\frac{1}{2}a_n+n-(n-j+1)-\frac{1}{2}t),&i>1.\\
\end{smallmatrix}\right.\]
\end{large}

Then \[t=\sum\limits_{u=n-i+1}^{n}a_u+
\sum\limits_{u=n-j+1}^{n-2}a_u+j+i-2.\]

These are the elements in $T_{n,\lambda}^D$. Hence $E_\lambda$ is an ACM bundle if and only if $n_l\geq 1$ for any integer $l\in[1,M_{n,\lambda}^D].$

\end{proof}

 \begin{ex}
 Let $E_\mu$ be an initialized homogeneous bundle with highest weight $\mu=\lambda_2+\lambda_3$ on $OG(5,11)=B_5/P(\alpha_5).$ We can get

 \[T_{5,\mu}^{B}=\left(\begin{matrix}
 1&2&4&6&7\\
 0&3&5&7&8\\
 0&0&7&9&10\\
 0&0&0&11&12\\
 0&0&0&0&13
 \end{matrix}\right),\]
$M_{5,\mu}^{B} =13$ and $n_l\geq1$ for any integer $l\in[1,13]$. So $E_{\mu}$ is an ACM bundle.
 \end{ex}
We generalize the above example in Corollary \ref{corequalnB}.
 \subsection{Applications}
 In this section, we always assume that $G$ is a simple Lie group of type $B$, $C$ or $D$.
 We use Theorem \ref{thm<n} and Theorem \ref{thm=n} to get some further results.
\begin{cor}\label{finite}
There are only finitely many irreducible homogeneous ACM bundles up to tensoring a line bundle over $G/P(\alpha_k)$. In particular, the moduli space of projective bundles produced by irreducible homogeneous ACM bundles consists of finite points.
\end{cor}
\begin{proof}
Let $E_\lambda$ be an irreducible homogeneous bundle with highest weight $\lambda=a_1\lambda_1+\dots+a_n\lambda_n.$
Without loss of generality, we may assume $E_\lambda$ is initialized i.e., $a_k=0$.
By Theorem \ref{thm<n} and \ref{thm=n}, if $E_\lambda$ is an ACM bundle, then $M_{k,\lambda}\leq \dim G/P(\alpha_k).$
Since $M_{k,\lambda}$ is the linear combination of $a_i$, where $a_i\geq0$, there are only finitely many choices of $a_i,$ i.e., there are only finitely many initialized irreducible homogeneous ACM bundles.
\end{proof}

Notice that the quadric
\[\textbf{Q}_n =\left\{\begin{matrix}B_m/P(\alpha_1),&\text{if }n=2m-1,\\
	D_m/P(\alpha_1),&\text{if }n=2m-2
\end{matrix}\right.\]
is the simplest rational homogeneous space of rank one besides Grassmannians.
Recall that there are some homogeneous bundles over $\textbf{Q}_n$ as the natural generalization of the universal subbundle and the dual of the quotient bundle over $\textbf{Q}_4\cong G(2,4)$. We call them \emph{spinor bundles}\cite{ottaviani1988spinor}. Specifically, when $n=2m-1$ there is only one spinor bundle induced by the irreducible represention with highest weight $\lambda_m$. While $n=2m-2$ there are two nonisomorphic spinor bundles induced by the irreducible representions with highest weights $\lambda_{m-1}$ and $\lambda_m$.

\begin{cor}\label{equals1}
With the notations as above.
\begin{enumerate}
\item[1.] The irreducible homogeneous ACM vector bundles over the quadric $\textbf{Q}_n$ are line bundles or spinor bundles up to tensoring a line bundle.
\item[2.] The irreducible homogeneous ACM vector bundles over $C_n/P(\alpha_1)$ are line bundles.
\end{enumerate}
 \end{cor}
 \begin{proof}
 For simplicity, we may assume $E_\lambda$ is initialized, which means $a_1=0$.

 1. If $n=2m-1$, then $\textbf{Q}_n=B_m/P(\alpha_1)$. Since $E_\lambda$ is an ACM bundle, we have \[M_{1,\lambda}^B= \sum\limits_{u=2}^{m-1}a_u+\sum\limits_{u=2}^{m}a_u+2m-2\leq 2m-1.\]
By calculation, it's easy to see  that these $a_i$ have only two choices: either $a_i=0~(1\le i\le m)$, which means that $E_\lambda$ is a line bundle, or  $a_m=1, a_i=0~(i\ne m)$, which means that $E_\lambda$ is the spinor bundle over $\textbf{Q}_n$.

If $n=2m-2$, $\textbf{Q}_n=D_m/P(\alpha_1)$. Since $E_\lambda$ is an ACM bundle, we have \[M_{1,\lambda}^D= \sum\limits_{u=2}^{m-2}a_u+\sum\limits_{u=2}^{m}a_u+2m-3\leq 2m-2.\]
  These $a_i$ have three choices:  $a_i=0~(1\le i\le m)$, which means that $E_\lambda$ is a line bundle, $a_{m-1}=1, a_i=0~(i\ne m-1)$, or $a_m=1, a_i=0~(i\ne m)$. The bundles given by the latter two cases are two nonisomorphic spinor bundles on $\textbf{Q}_n$.

  2. The dimension of $C_n/P(\alpha_1)$ is $2n-1$. Since $E_\lambda$ is an ACM bundle, we still have \[M_{1,\lambda}^C= 2\sum\limits_{u=2}^{n}a_u+2n-1\leq 2n-1.\]
  The only choice of $a_i$ is $a_i=0~(1\le i\le n)$ which means that $E_\lambda$ is a line bundle.
  \end{proof}
  \begin{rmk}
\emph{There is a strong version of Corollary \ref{equals1}.1 (see \cite{buchweitz1987cohen}\cite{knorrer1987cohen}\cite{ottaviani1989some}).}
  \end{rmk}

For general cases, the step matrices are complicated. However, we can use succinct forms to characterize some irreducible initialized homogeneous ACM bundles with special highest weights over isotropic Grassmannians of types $B$, $C$ and $D$.
We first state the theorem for $B_n/P(\alpha_k)$.

\begin{cor}\label{corBlessn}
Let $E_\lambda$ be an initialized irreducible homogeneous bundle over $B_n/P(\alpha_k)$ for $1<k<n.$
\begin{enumerate}
\item[1.] Suppose $\lambda=a_1\lambda_1+\dots+a_{k-1}\lambda_{k-1}.$ If $0\leq a_{1+k-i}\leq 2n-2k$ for all $2\leq i\leq k$, then $E_\lambda$ is an ACM bundle.
\item[2.] Suppose $\lambda=a_{k+1}\lambda_{k+1}+\dots+a_{n}\lambda_{n}.$ Then $E_\lambda$ is an ACM bundle if and only if $0\leq a_{i}\leq k-1$ for all $k+1\leq i\leq n-1$, and $0\leq a_n\leq 2k-1$.
\end{enumerate}
\end{cor}
\begin{proof}
By Theorem \ref{thm<n}, proving that $E_\lambda$ is an ACM bundle is equivalent to proving all integer values between 1 and $M_{k,\lambda}^B$ appear as entries of $T_{k,\lambda}^B$.

1. Since $a_u=0$ for
$u\geq k$, then we can write $T_{k,\lambda}^B=(P_{k,\lambda}^B,Q_{k,\lambda}^B,R_{k,\lambda}^B)$ as follows.

$$P_{k,\lambda}^B= \left(\begin{matrix}
 1& 2&\dots&  n-k\\
 \vdots&\vdots&\ddots&\vdots\\
\bm{\sum\limits_{u=1+k-i}^{k-1}a_u+i}&\sum\limits_{u=1+k-i}^{k-1}a_u+i+1&\dots&\sum\limits_{u=1+k-i}^{k-1}a_u+i+n-k-1\\
\vdots&\vdots&\ddots&\vdots\\
\sum\limits_{u=1}^{k-1}a_u+k&\sum\limits_{u=1}^{k-1}a_u+k+1&\dots&\sum\limits_{u=1}^{k-1}a_u+n-1\\
   \end{matrix}   \right), $$
where we emphasize $$p_{i1}=\sum\limits_{u=1+k-i}^{k-1}a_u+i;$$

   \begin{footnotesize}
    $$ Q_{k,\lambda}^B= \left(\begin{matrix}
 n-k+1& n-k+2&\dots& 2n-2k\\
 \vdots&\vdots&\ddots&\vdots\\
 \sum\limits_{u=2+k-i}^{k-1}a_u+n-k+i-1& \sum\limits_{u=2+k-i}^{k-1}a_u+n-k+i&\dots& \bm{\sum\limits_{u=2+k-i}^{k-1}a_u+2n-2k+i-2}\\
\vdots&\vdots&\ddots&\vdots\\
\sum\limits_{u=1}^{k-1}a_u+n &\sum\limits_{u=1}^{k-1}a_u+n+1&\dots&\sum\limits_{u=1}^{k-1}a_u+2n-k-1\\
   \end{matrix}    \right),$$
   \end{footnotesize}
where we emphasize  $$q_{i-1,n-k}=\sum\limits_{u=2+k-i}^{k-1}a_u+2n-2k+i-2;$$
and $$r_{i-1,i}=\sum\limits_{u=2+k-i}^{k-1}a_u+\frac{a_{1+k-i}}{2}+n-k+i-1.$$

Notice first that the entries of each row of matrices $P_{k,\lambda}^B$ and $Q_{k,\lambda}^B$ are consecutive integers and $p_{i,n-k}=q_{i,1}-1$ for fixed integer $i~(1\le i\le k)$.

If $a_{k-1}\leq 2n-2k-1$, then $p_{21}=a_{k-1}+2\leq 2n-2k+1$. It's easy to see that any integer $l\in[1,p_{21}]$ appears as an entry of $P_{k,\lambda}^B$ or $Q_{k,\lambda}^B$. If $a_{k-1}=2n-2k$, then $p_{21}=a_{k-1}+2=2n-2k+2$. Any integer $l\in[1,p_{21}-2=2n-2k]$ appears as an entry of $P_{k,\lambda}^B$ or $Q_{k,\lambda}^B$. In this case, $p_{21}-1= 2n-2k+1=\frac{2n-2k }{2}+n-k+1=r_{12}$ appears as an entry of $R_{k,\lambda}^B$. It follows that any integer $l\in [1,p_{21}]$ appears as an entry of $T_{k,\lambda}^B$ as long as $a_{k-1}\leq 2n-2k $.

Arguing in the same way if $a_{1+k-i}\le 2n-2k-1~(2<i\le k)$, we have
 $$p_{i1}=\sum\limits_{u=1+k-i}^{k-1}a_u+i\le\sum\limits_{u=2+k-i}^{k-1}a_u+2n-2k+i-1=q_{i-1,n-k}+1.$$
It's obvious that any integer $l\in [p_{i-1,1},p_{i1}]$ appears as an entry of $P_{k,\lambda}^B$ or $Q_{k,\lambda}^B$.

 If $a_{1+k-i}=2n-2k $, then any integer $l\in [p_{i-1,1},p_{i1}-2=q_{i-1,n-k}]$ appears as an entry of $P_{k,\lambda}^B$ or $Q_{k,\lambda}^B$. Notice that
$$p_{i1}-1=q_{i-1,n-k}+1= \sum\limits_{u=2+k-i}^{k-1}a_u+\frac{2n-2k}{2}+n-k-1+i=r_{i-1,i}.$$
So any integer $l\in [p_{i-1,1},p_{i1}]$ appears as an entry of $T_{k,\lambda}^B$ as long as $a_{1+k-i}\leq 2n-2k$.

To sum up, if $0\leq a_{1+k-i}\leq 2n-2k$ for all $2\leq i\leq k$, then any integer $l\in [1,p_{k1}]$ appears as an entry of $T_{k,\lambda}^B$. Since $M_{k,\lambda}^B=q_{k,n-k}$, all integer values between 1 and $M_{k,\lambda}^B$ appear as entries of $T_{k,\lambda}^B$.

2. We first note that $a_u=0$ for $u\leq k.$ Then $T_{k,\lambda}^B=(P_{k,\lambda}^B,Q_{k,\lambda}^B,R_{k,\lambda}^B)$, where
$$P_{k,\lambda}^B= \left(\begin{matrix}
 1& \dots &\bm{\sum\limits_{u=k+1}^{k-1+j}a_u+j}&\dots& \sum\limits_{u=k+1}^{n-1}a_u+n-k\\
 2&\dots & \sum\limits_{u=k+1}^{k-1+j}a_u+j+1&\dots&\sum\limits_{u=k+1}^{n-1}a_u+n-k+1\\
\vdots&\vdots &\vdots&\ddots&\vdots\\
 k& \dots &\bm{\sum\limits_{u=k+1}^{k-1+j}a_u+j+k-1}&\dots&\sum\limits_{u=k+1}^{n-1}a_u+n-1\\
   \end{matrix}   \right), $$

    $$ Q_{k,\lambda}^B= \left(\begin{smallmatrix}
\sum\limits_{u=k+1}^{n}a_u+n-k+1&\dots&\bm{\sum\limits_{u=k+1}^{n}a_u+\sum\limits_{u=n+1-j}^{n-1}a_u+n-k+j}&\dots&\sum\limits_{u=k+1}^{n}a_u+\sum\limits_{u=k+1}^{n-1}a_u+2n-2k\\
\sum\limits_{u=k+1}^{n}a_u+n-k+2&\dots&\sum\limits_{u=k+1}^{n}a_u+\sum\limits_{u=n+1-j}^{n-1}a_u+n-k+j+1&\dots&\sum\limits_{u=k+1}^{n}a_u+\sum\limits_{u=k+1}^{n-1}a_u+2n-2k+1\\
\vdots&\dots&\vdots&\ddots&\vdots\\
\sum\limits_{u=k+1}^{n}a_u+n &\dots&\bm{\sum\limits_{u=k+1}^{n}a_u+\sum\limits_{u=n+1-j}^{n-1}a_u+n+j-1}&\dots&\sum\limits_{u=k+1}^{n}a_u+\sum\limits_{u=k+1}^{n-1}a_u+2n-k-1\\
   \end{smallmatrix}    \right)$$
and
$$ r_{ij}=\sum\limits_{u=k+1}^{n-1}a_u+n-k-1+\frac{a_n+i+j+1}{2}.$$


 $(\impliedby)$ Notice first that the entries of all columns of matrices $P_{k,\lambda}^B$ and $Q_{k,\lambda}^B$ are consecutive integers. If $0\leq a_{i}\leq k-1$ for all $k+1\leq i\leq n-1$, then for any integer $j~(2\le j\le n-k)$,
$$p_{1j}=\sum\limits_{u=k+1}^{k-1+j}a_u+j\le \sum\limits_{u=k+1}^{k-2+j}a_u+k-1+j=p_{k,j-1}+1$$
and
$$q_{1j}=\sum\limits_{u=k+1}^{n}a_u+\sum\limits_{u=n+1-j}^{n-1}a_u+n-k+j\le \sum\limits_{u=k+1}^{n}a_u+\sum\limits_{u=n+2-j}^{n-1}a_u+n-1+j=q_{k,j-1}+1.$$

Similar to the proof of the first statement in Proposition \ref{corBlessn}, we find that any integer $l\in [1,p_{k,n-k}]$ appears as an entry of $P_{k,\lambda}^B$ and any integer $l\in [q_{11},q_{k,n-k}=M_{k,\lambda}^B]$ appears as an entry of $Q_{k,\lambda}^B$. Hence, in order to prove $E_\lambda$ being an ACM bundle, it suffices to show that any integer $l\in [p_{k,n-k},q_{11}]$ appears as an entry of $T_{k,\lambda}^B$. This is obviously true if $0\leq a_n\leq k-1$, because in this case $q_{11}=\sum\limits_{u=k+1}^{n}a_u+n-k+1\le\sum\limits_{u=k+1}^{n-1}a_u+n=p_{k,n-k}+1$.

Suppose $a_n=k+m$, where $0\leq m\leq k-1$. Then
$$q_{11}=\sum\limits_{u=k+1}^{n}a_u+n-k+1=\sum\limits_{u=k+1}^{n-1}a_u+n+m+1=p_{k,n-k}+m+2.$$
Therefore, in order to show that any integer $l\in [p_{k,n-k},q_{11}]$ appears as an entry of $T_{k,\lambda}^B$, we only need to show that for any
$l\in [1,m+1]$, $p_{k,n-k}+l$ appears as an entry of $T_{k,\lambda}^B$, which is true since \[p_{k,n-k}+l=\sum\limits_{k+1}^{n-1}a_u+n-1+l=r_{l,l+k-m-1}.\]

$(\implies)$ If for some integer $j~(2\le j\le n-k)$, $a_{k-1+j}>k-1$, then
$$p_{k,j-1}+1=\sum\limits_{u=k+1}^{k-2+j}a_u+k-1+j<p_{1j}=\sum\limits_{u=k+1}^{k-1+j}a_u+j<r_{11}<q_{11}.$$
Hence $p_{k,j-1}+1$ does not appear as an entry of $T_{k,\lambda}^B$. Similarly, if $a_n>2k-1$,
\[p_{k,n-k}+1=\sum\limits_{u=k+1}^{n-1}a_u+n<\sum\limits_{u=k+1}^{n-1}a_u+n-k+\frac{a_n+1}{2}=r_{11}<q_{11}.\]
Then integer $p_{k,n-k}+1$ can not be an entry of $T_{k,\lambda}^B$.
\end{proof}

\begin{cor}\label{corequalnB}
Let $E_\lambda$ be an initialized irreducible homogeneous bundle over $B_n/P(\alpha_n)$ with $\lambda=a_2\lambda_2+...+a_{n-2}\lambda_{n-2}$.
 If $0\leq a_i\leq 1$ for all $2\leq i\leq n-2$, then $E_\lambda$ is an ACM bundle.
\end{cor}
\begin{proof}
   Since $a_1=a_{n-1}=a_n=0,$
  $$T_{n,\lambda}^B=\left(\begin{matrix}
 1 & 2&\dots &\sum\limits_{u=2}^{n-2}a_u+n-1 & \sum\limits_{u=2}^{n-2}a_u+n\\
0& 3&\dots& \sum\limits_{u=2}^{n-2}a_u+n &\sum\limits_{u=2}^{n-2}a_u+n+1\\
\vdots&\vdots&\ddots&\vdots&\vdots\\
0& 0 &\dots &2\sum\limits_{u=3}^{n-2}a_u+a_2+2n-4&2\sum\limits_{u=3}^{n-2}a_u+a_2+2n-3\\
0& 0 &\dots &2\sum\limits_{u=2}^{n-2}a_u+2n-3&2\sum\limits_{u=2}^{n-2}a_u+2n-2\\
0& 0 &\dots&0&2\sum\limits_{u=2}^{n-2}a_u+2n-1\\
   \end{matrix}   \right). $$

When $j\geq 2$, $t_{2j}-t_{1j}=1 ~\text{and} ~ t_{i,j+1}-t_{ij}=a_{n-j}+1~(i=1,2). $
Then $t_{1,j+1}=t_{2,j}+a_{n-j}$. If $0\leq a_{n-j}\leq 1~( 2\leq j\leq n-2)$, then $t_{1,j+1}=t_{2j}$ or $t_{1,j+1}=t_{2j}+1.$
  Hence if we consider the first two rows of the step matrix, then it is easy to see that all integers between $1$ and $\sum\limits_{u=2}^{n-2}a_u+n+1$ are in $T_{n,\lambda}^B$. Meanwhile, notice that $t_{i+1,n-1}=t_{i,n}+a_{n-i}~(i\leq n-2)$.
   If $0\leq a_{n-i}\leq 1~( 2\leq i\leq n-2)$, then $t_{i+1,n-1}=t_{i,n}$ or $t_{i+1,n-1}=t_{i,n}+1.$ By considering the last two columns of the step matrix, all integers between $\sum\limits_{u=2}^{n-2}a_u+n-1$ and $2\sum\limits_{u=2}^{n-2}a_u+2n-1$ are in $T_{n,\lambda}^B$. In consequence, all integers between $1$ and $M_{n,\lambda}^B=2\sum\limits_{u=2}^{n-2}a_u+2n-1$  are in $T_{n,\lambda}^B$. Hence $E_\lambda$ is an ACM bundle.

\end{proof}
For $C_n/P(\alpha_k)$, we have the following consequence.

\begin{cor}\label{cor<=nC}
Let $E_\lambda$ be an initialized irreducible homogeneous bundle over $C_n/P(\alpha_k)$ for $1<k\leq n$.
\begin{enumerate}
\item[1.] Suppose $\lambda=a_1\lambda_1+\dots+a_{k-1}\lambda_{k-1}.$ If $0\leq a_{1+k-i}\leq 2n-2k+1$ for all $2\leq i\leq k$, then $E_\lambda$ is an ACM bundle.
\item[2.] Suppose $\lambda=a_{k+1}\lambda_1+\dots+a_{n}\lambda_{n}.$ $E_\lambda$ is an ACM bundle if and only if $0\leq a_{i}\leq k-1$ for all $k+1\leq i\leq n$.
\end{enumerate}
\end{cor}

\begin{proof}
1.
We first prove the case where $k<n$.
Since $a_u=0$ for
$u\geq k$, then we can write $T_{k,\lambda}^C=(P_{k,\lambda}^C,Q_{k,\lambda}^C,R_{k,\lambda}^C)$ as follows.

$$P_{k,\lambda}^C= \left(\begin{matrix}
 1& 2&\dots&  n-k\\
 \vdots&\vdots&\ddots&\vdots\\
\bm{\sum\limits_{u=1+k-i}^{k-1}a_u+i}&\sum\limits_{u=1+k-i}^{k-1}a_u+i+1&\dots&\sum\limits_{u=1+k-i}^{k-1}a_u+i+n-k-1\\
\vdots&\vdots&\ddots&\vdots\\
\sum\limits_{u=1}^{k-1}a_u+k&\sum\limits_{u=1}^{k-1}a_u+k+1&\dots&\sum\limits_{u=1}^{k-1}a_u+n-1\\
   \end{matrix}   \right),$$
where we emphasize $$p_{i1}=\sum\limits_{u=1+k-i}^{k-1}a_u+i;$$
   \begin{footnotesize}
    $$ Q_{k,\lambda}^C= \left(\begin{matrix}
 n-k+2& n-k+3&\dots& 2n-2k+1\\
 \vdots&\vdots&\ddots&\vdots\\
 \sum\limits_{u=2+k-i}^{k-1}a_u+n-k+i & \sum\limits_{u=2+k-i}^{k-1}a_u+n-k+i+1&\dots& \bm{\sum\limits_{u=2+k-i}^{k-1}a_u+2n-2k+i-1}\\
\vdots&\vdots&\ddots&\vdots\\
\sum\limits_{u=1}^{k-1}a_u+n+1 &\sum\limits_{u=1}^{k-1}a_u+n+2&\dots&\sum\limits_{u=1}^{k-1}a_u+2n-k\\
   \end{matrix}    \right),$$
   \end{footnotesize}
where we emphasize $$q_{i-1,n-k}=\sum\limits_{u=2+k-i}^{k-1}a_u+2n-2k+i-1;$$

$$r_{i,i}=\sum\limits_{u=1+k-i}^{k-1}a_u+n-k+i$$ and
$$r_{i-1,i}=\sum\limits_{u=2+k-i}^{k-1}a_u+\frac{a_{1+k-i}-1}{2}+n-k+i.$$
Notice that the entries of all rows of matrices $P_{k,\lambda}^C$ and $Q_{k,\lambda}^C$ are consecutive integers and $p_{i,n-k}+2=r_{i,i}+1=q_{i1}$ for fixed integer $i~(1\le i\le k)$.

As the proof of the first statement of Corollary \ref{corBlessn}, if $a_{1+k-i}\le 2n-2k~(2\le i\le k)$, then
 $$p_{i1}=\sum\limits_{u=1+k-i}^{k-1}a_u+i\le\sum\limits_{u=2+k-i}^{k-1}a_u+2n-2k+i=q_{i-1,n-k}+1.$$
It's obvious that any integer $l\in [p_{i-1,1},p_{i1}]$ appears as an entry of $P_{k,\lambda}^C$ or $Q_{k,\lambda}^C$.

 If $a_{1+k-i}=2n-2k+1$, then any integer $l\in [p_{i-1,1},p_{i1}-2=q_{i-1,n-k}]$ appears as an entry of $P_{k,\lambda}^C$ or $Q_{k,\lambda}^C$.  We also have
$$p_{i1}-1=q_{i-1,n-k}+1= \sum\limits_{u=2+k-i}^{k-1}a_u+\frac{2n-2k}{2}+n-k+i=r_{i-1,i}.$$
It follows that any integer $l\in [p_{i-1,1},p_{i1}]$ appears as an entry of $T_{k,\lambda}^C$ as long as $a_{1+k-i}\leq 2n-2k+1$.

For $k=n$, from the definition of $T_{n,\lambda}^C$ (See Definition \ref{defequaln}), we have
$$t_{i,i}=\sum\limits_{u=1+n-i}^{n}a_u+i$$ and
$$t_{i-1,i}=\sum\limits_{u=2+n-i}^{n}a_u+\frac{a_{1+n-i}-1}{2}+i.$$
If $a_{1+n-i}=0,$ then $t_{i,i}-t_{i-1,i-1}=1$.
If $a_{1+n-i}=1,$ then $t_{i,i}=t_{i-1,i}+1=t_{i-1,i-1}+2$. If we consider $\{t_{ii}\}$ and $\{t_{i-1,i}\}$, then it is easy to see that any integer $l\in [1,M_{n,\lambda}^C=t_{nn}]$ appears as an entry of $T_{n,\lambda}^C$.

Hence if $0\leq a_{1+k-i}\leq 2n-2k+1$ for all $2\leq i\leq k$, then any integer $l\in [1,M_{k,\lambda}^C]$ appears as an entry of $T_{k,\lambda}^C$.

2. Since $a_{u}=0~(u\leq k)$, $T_{k,\lambda}^C=(P_{k,\lambda}^C,Q_{k,\lambda}^C,R_{k,\lambda}^C)$ where
$$P_{k,\lambda}^C= \left(\begin{matrix}
 1& \dots &\bm{\sum\limits_{u=k+1}^{k-1+j}a_u+j}&\dots& \sum\limits_{u=k+1}^{n-1}a_u+n-k\\
 2&\dots & \sum\limits_{u=k+1}^{k-1+j}a_u+j+1&\dots&\sum\limits_{u=k+1}^{n-1}a_u+n-k+1\\
\vdots&\vdots &\vdots&\ddots&\vdots\\
 k& \dots &\bm{\sum\limits_{u=k+1}^{k-1+j}a_u+j+k-1}&\dots&\sum\limits_{u=k+1}^{n-1}a_u+n-1\\
   \end{matrix}   \right), $$

\begin{normalsize}
    $$   Q_{k,\lambda}^C=\left(\begin{smallmatrix}
\sum\limits_{u=k+1}^{n-1}a_u+2a_n+n-k+2&\dots&\bm{\sum\limits_{u=k+1}^{n}a_u+\sum\limits_{u=n+1-j}^{n}a_u+n-k+j+1}&\dots&2\sum\limits_{u=k+1}^{n}a_u+2n-2k+1\\
\sum\limits_{u=k+1}^{n-1}a_u+2a_n+n-k+3&\dots&\sum\limits_{u=k+1}^{n}a_u+\sum\limits_{u=n+1-j}^{n}a_u+n-k+j+2&\dots&2\sum\limits_{u=k+1}^{n}a_u+2n-2k+3\\
\vdots&\dots&\vdots&\ddots&\vdots\\
\sum\limits_{u=k+1}^{n-1}a_u+2a_n+n+1 &\dots&\bm{\sum\limits_{u=k+1}^{n}a_u+\sum\limits_{u=n+1-j}^{n}a_u+n+j}&\dots&2\sum\limits_{u=k+1}^{n}a_u+2n-k\\
   \end{smallmatrix}    \right)$$

\end{normalsize}

and
$$ r_{ii}= \sum\limits_{u=k+1}^{n}a_u+n-k+i.$$

 $(\impliedby)$ Similar to the proof of the second statement in Corollary \ref{corBlessn},  if $0\leq a_i\leq k-1$ for all $k+1\leq i\leq n-1$, then for any integer $l$ in $[1,p_{k,n-k}]$ and $[q_{11},q_{k,n-k}]$ appears in $T_{k,\lambda}^C$. Furthermore, if $0\leq a_n\leq k-1$, then $r_{11}\leq p_{k,n-k}+1$ and $q_{11}\leq r_{kk}+1$. It is easy to see that the diagonal of $R_{k,\lambda}^C$ are consecutive integers.
Hence we show that any integer $l\in[1,q_{k,n-k}=M_{k,\lambda}^C]$ appears in $T_{k,\lambda}^C.$

$(\implies)$ If for some integer $j~(2\le j<n-k+1)$, $a_{k-1+j}>k-1$, then
$$p_{k,j-1}+1=\sum\limits_{u=k+1}^{k-2+j}a_u+k-1+j<p_{1j}=\sum\limits_{u=k+1}^{k-1+j}a_u+j<r_{11}<q_{11},$$

Hence $p_{k,j-1}+1$ would not appear as an entry of $T_{k,\lambda}^C$. Similarly, if $a_{n}>k-1$,
then $$p_{k,n-k}+1=\sum\limits_{u=k+1}^{n-1}a_u+n<r_{11}=\sum\limits_{u=k+1}^{n}a_u+n-k+1<q_{11}.$$
So $p_{k,n-k}+1$ would not appear as an entry of $T_{k,\lambda}^C$.
\end{proof}

For type $D_n/P(\alpha_k)$, we have the following two corollaries.

\begin{cor}\label{cor<nD}
Let $E_\lambda$ be an initialized irreducible homogeneous bundle over $D_n/P(\alpha_k)$ for $1<k<n-1$.
\begin{enumerate}
\item[1.] Suppose $\lambda=a_1\lambda_1+\dots+a_{k-1}\lambda_{k-1}.$ If $0\leq a_{1+k-i}\leq 2n-2k-1$ for all $2\leq i\leq k$, then $E_\lambda$ is an ACM bundle.
\item[2. ]Suppose $\lambda=a_{k+1}\lambda_1+\dots+a_{n}\lambda_{n}.$ $E_\lambda$ is an ACM bundle if and only if $0\leq a_i\leq k-1$ for all $k+1\leq i\leq n-2$, and

$(a)$ $ \left\{\begin{matrix}0\leq a_{n-1}\leq k-1,\\
0\leq a_n-a_{n-1}\leq 2k-1,
\end{matrix}\right.$   \ \ or   \ \  $(b)$ $ \left\{\begin{matrix}0\leq a_{n}\leq k-1,\\
0\leq a_{n-1}-a_{n}\leq 2k-1.
\end{matrix}\right.$
\end{enumerate}

\end{cor}
\begin{proof}
1. Since $a_u=0$ for
$u\geq k$, then we can write $T_{k,\lambda}^C=(P_{k,\lambda}^C,Q_{k,\lambda}^C,R_{k,\lambda}^C)$ as follows.
$$P_{k,\lambda}^D= \left(\begin{matrix}
 1& 2&\dots&  n-k\\
 \vdots&\vdots&\ddots&\vdots\\
\bm{\sum\limits_{u=1+k-i}^{k-1}a_u+i}&\sum\limits_{u=1+k-i}^{k-1}a_u+i+1&\dots&\sum\limits_{u=1+k-i}^{k-1}a_u+i+n-k-1\\
\vdots&\vdots&\ddots&\vdots\\
\sum\limits_{u=1}^{k-1}a_u+k&\sum\limits_{u=1}^{k-1}a_u+k+1&\dots&\sum\limits_{u=1}^{k-1}a_u+n-1\\
   \end{matrix}   \right),$$
where we emphasize $$p_{i1}=\sum\limits_{u=1+k-i}^{k-1}a_u+i;$$
  \begin{footnotesize}
    $$ Q_{k,\lambda}^D=\left(\begin{matrix}
n-k& n-k+1&\dots& 2n-2k-1\\
\vdots&\vdots&\ddots&\vdots\\
\sum\limits_{u=2+k-i}^{k-1}a_u+n-k+i-2& \sum\limits_{u=2+k-i}^{k-1}a_u+n-k+i-1&\dots&\bm{ \sum\limits_{u=2+k-i}^{k-1}a_u+2n-2k+i-3}\\
\vdots&\vdots&\ddots&\vdots\\
\sum\limits_{u=1}^{k-1}a_u+a_n+n-1 &\sum\limits_{u=1}^{k-1}a_u+n&\dots&\sum\limits_{u=1}^{k-1}a_u+2n-k-2\\
   \end{matrix}    \right)$$
   \end{footnotesize}
where we emphasize $$q_{i-1,n-k}=\sum\limits_{u=2+k-i}^{k-1}a_u+2n-2k+i-3$$

and
\[r_{i-1,i}=\sum\limits_{u=2+k-i}^{k-1}a_u+\frac{a_{1+k-i}-3}{2}+n-k+i.\]

Notice first that the entries of all rows of matrices $P_{k,\lambda}^D$ and $Q_{k,\lambda}^D$ are consecutive integers and $p_{i,n-k} =q_{i,1}$ for fixed integer $i~(1\le i\le k)$.

As the proof of the first statement of Corollary \ref{corBlessn}, if $a_{1+k-i}\le 2n-2k-2~(2\le i\le k)$, then
 $$p_{i1}=\sum\limits_{u=1+k-i}^{k-1}a_u+i\le\sum\limits_{u=2+k-i}^{k-1}a_u+2n-2k+i-2=q_{i-1,n-k}+1.$$
It's obvious that any integer $l\in [p_{i-1,1},p_{i1}]$ appears as an entry of $P_{k,\lambda}^D$ or $Q_{k,\lambda}^D$.

 If $a_{1+k-i}=2n-2k-1$, then any integer $l\in [p_{i-1,1},p_{i1}-2=q_{i-1,n-k}]$ appears as an entry of $P_{k,\lambda}^D$ or $Q_{k,\lambda}^D$. We also have
$$p_{i1}-1=q_{i-1,n-k}+1= \sum\limits_{u=2+k-i}^{k-1}a_u+\frac{2n-2k-4}{2}+n-k+i=r_{i-1,i}.$$
It follows that any integer $l\in [p_{i-1,1},p_{i1}]$ appears as an entry of $T_{k,\lambda}^D$ as long as $a_{1+k-i}\leq 2n-2k-1$.

To sum up, if $0\leq a_{1+k-i}\leq 2n-2k-1$ for all $2\leq i\leq k$, then any integer $l\in [1,p_{k1}]$ appears as an entry of $T_{k,\lambda}^D$. Since $M_{k,\lambda}^D=q_{k,n-k}$, all integer values between 1 and $M_{k,\lambda}^D$ appear as entries of $T_{k,\lambda}^D$.

2. Since $a_u=0~(u\leq k)$,
$T_{k,\lambda}^D=(P_{k,\lambda}^D,Q_{k,\lambda}^D,R_{k,\lambda}^D)$, where
$$P_{k,\lambda}^D= \left(\begin{matrix}
 1& \dots &\bm{\sum\limits_{u=k+1}^{k-1+j}a_u+j}&\dots& \sum\limits_{u=k+1}^{n-1}a_u+n-k\\
 2&\dots & \sum\limits_{u=k+1}^{k-1+j}a_u+j+1&\dots&\sum\limits_{u=k+1}^{n-1}a_u+n-k+1\\
\vdots&\vdots &\vdots&\ddots&\vdots\\
 k& \dots &\bm{\sum\limits_{u=k+1}^{k-1+j}a_u+j+k-1}&\dots&\sum\limits_{u=k+1}^{n-1}a_u+n-1\\
   \end{matrix}   \right), $$

    $$ Q_{k,\lambda}^D=\left(\begin{smallmatrix}
\sum\limits_{u=k+1}^{n-2}a_u+a_n+n-k&\dots&\bm{\sum\limits_{u=k+1}^{n-2}a_u+\sum\limits_{u=n+1-j}^{n}a_u+n-k+j-1}&\dots&\sum\limits_{u=k+1}^{n}a_u+\sum\limits_{u=k+1}^{n-2}a_u+2n-2k-1\\
\sum\limits_{u=k+1}^{n-2}a_u+a_n+n-k+1&\dots& \sum\limits_{u=k+1}^{n-2}a_u+\sum\limits_{u=n+1-j}^{n}a_u+n-k+j&\dots&\sum\limits_{u=k+1}^{n}a_u+\sum\limits_{u=k+1}^{n-2}a_u +2n-2k\\
\vdots&\dots&\vdots&\ddots&\vdots\\
\sum\limits_{u=k+1}^{n-2}a_u+a_n+n-1 &\dots&\bm{\sum\limits_{u=k+1}^{n-2}a_u+\sum\limits_{u=n+1-j}^{n}a_u+n+j-2}&\dots&\sum\limits_{u=k+1}^{n}a_u+\sum\limits_{u=k+1}^{n-2}a_u+2n-k-2\\
   \end{smallmatrix}    \right),$$

\[r_{i-1,i}=n-k+i-1+\sum\limits_{u=1+k}^{n-1}a_u+ \frac{1}{2}(a_{n}-a_{n-1}-1)\] and

\[r_{i-2,i}=n-k+i-2+\sum\limits_{u=1+k}^{n-1}a_u+ \frac{1}{2}( a_{n}-a_{n-1}).\]

$(\impliedby)$ Without loss of generality, we only need to consider case (a).
Similar to the proof of the second statement of Corollary \ref{corBlessn}, if $0\leq a_i\leq k-1$ for all $k+1\leq i\leq n-1$, then for any integer $l$ in $[1,p_{k,n-k}]$ and $[q_{11},q_{k,n-k}]$ appears in $T_{k,\lambda}^D$. By assumption, $0\leq a_n-a_{n-1}\leq 2k-1$. We divide into the following two cases.

If $a_{n}-a_{n-1}$ is even, then $\{r_{i-2,i}\}$ are consecutive integers.
In this case, $a_n-a_{n-1}\leq 2k-2$. If $k=2$, then \[q_{11} =\sum\limits_{u=3}^{n-2}a_u+a_n+n-2=\leq \sum\limits_{u=3}^{n-1}a_u+n=p_{2,n-2}+1.\]
If $k>2$, then we have \[r_{13}=n-k+1+\sum\limits_{u=1+k}^{n-1}a_u+ \frac{1}{2}( a_{n}-a_{n-1})\leq\sum\limits_{u=k+1}^{n-1}a_u+n=p_{k,n-k}+1 \]
and \[q_{11}=\sum\limits_{u=k+1}^{n-2}a_u+a_n+n-k\leq n-1+\sum\limits_{u=1+k}^{n-1}a_u+\frac{1}{2}(a_{n}-a_{n-1})=r_{k-2,k}+1.\]

If $a_{n}-a_{n-1}$ is odd, then $\{r_{i-1,i}\}$ are consecutive integers.
Since $a_n-a_{n-1}\leq 2k-1$, we have \[r_{12}=n-k+1+\sum\limits_{u=1+k}^{n-1}a_u+ \frac{1}{2}( a_{n}-a_{n-1}-1)\leq\sum\limits_{u=k+1}^{n-1}a_u+n=p_{k,n-k}+1 \]
and \[q_{11}=\sum\limits_{u=k+1}^{n-2}a_u+a_n+n-k\leq n+\sum\limits_{u=1+k}^{n-1}a_u+\frac{1}{2}(a_{n}-a_{n-1}-1)=r_{k-1,k}+1.\]
So it is easy to show that any integer $l\in[1,M_{\lambda,k}^D=q_{k,n-k}]$ appears as an entry of $T_{\lambda,k}^D$ in the above two cases.

$(\implies)$ For case $(a)$, if for some integer $j~(2\le j< n-k+1)$, $a_{k-1+j}>k-1$, then
$p_{k,j-1}+1<p_{1j}<q_{11}$ and

\begin{equation*}
\begin{aligned}
r_{12}-(p_{k,j-1}+1)& =n-k+1+\sum\limits_{u=1+k}^{n-2}a_u+ \frac{1}{2}( a_{n}+a_{n-1}-1) -(\sum\limits_{u=k+1}^{k-2+j}a_u+j+k-1)\\
         &=n-2k-j+\frac{3}{2}+\sum\limits_{u=k+j}^{n-2}a_u+a_{k-1+j}+\frac{1}{2}(a_{n-1}+a_n)\\
         &>n-k-j+\frac{1}{2}+\sum\limits_{u=k+j}^{n-2}a_u+\frac{1}{2}(a_{n-1}+a_n)>0.
\end{aligned}
\end{equation*}

Hence $p_{k,j-1}+1$ would not appear as an entry of $T_{k,\lambda}^D$. If $a_{n}-a_{n-1}>2k-1,$ then $$r_{12}-(p_{k,n-k}+1)=\frac{1}{2}(-a_{n-1}+a_{n}+1-2k)>0,$$
i.e., $p_{k,n-k}+1<r_{12}<q_{11}.$

Hence $p_{k,n-k}+1$ would not appear as an entry of $T_{k,\lambda}^D$. Similar to prove case $(b)$.

\end{proof}
\begin{cor}\label{cor=nD}
Let $E_\lambda$ be an initialized irreducible homogeneous bundle over $D_n/P(\alpha_n)$ with $\lambda=a_1\lambda_1+\dots+a_{n-1}\lambda_{n-1}$.

 1. If $0\leq a_i\leq 1$ for all $3\leq i\leq n-3$ and $a_1=a_2=a_{n-1}=a_{n-2}=0$, then $E_\lambda$ is an ACM bundle.

  2. If $a_i=0$ for all $2\leq i\leq n-2$ and $0\leq a_1,a_{n-1}\leq n-4$, then $E_\lambda$ is an ACM bundle.
\end{cor}
\begin{proof}
1. If $a_1=a_2=a_{n-1}=a_{n-2}=0$, then
   $$ T_{n,\lambda}^D=\left(\begin{matrix}
0&1& 2 &a_{n-3}+3& \dots& \sum\limits_{u=3}^{n-3}a_u+n-2 & \sum\limits_{u=3}^{n-3}a_u+n-1 \\
0&0& 3 &a_{n-3}+4&\dots&\sum\limits_{u=3}^{n-3}a_u+n-1& \sum\limits_{u=3}^{n-3}a_u+n \\
  \vdots&\vdots&\ddots&\ddots&\dots&\vdots&\vdots\\  \vdots&\vdots&\vdots&\ddots&\ddots&\vdots&\vdots\\
   0&0&0&0&\ddots&   2\sum\limits_{u=3}^{n-3}a_u+2n-5& 2\sum\limits_{u=3}^{n-3}a_u+2n-4\\
 0&0&0&0&\dots&0&  2\sum\limits_{u=3}^{n-3}a_u+2n-3\\
 0&0&0&0&\dots&0& 0\\
   \end{matrix}   \right). $$

   As the proof in Corollary \ref{corequalnB}, $t_{1,j+1}=t_{2,j}+a_{n-j}~(j\geq3)$. If $0\leq a_{n-j}\leq1~(3\leq j\leq n-3),$ then $t_{1,j+1}=t_{2j}$ or $t_{1,j+1}=t_{2j}+1.$ Meanwhile, $t_{i+1,n-1}=t_{i,n}+a_{n-i}~(i\leq n-3).$ If $0\leq a_{n-i}\leq 1~(3\leq i\leq n-3)$, then $t_{i+1,n-1}=t_{i,n}$ or $t_{i+1,n-1}=t_{i,n}+1.$ By considering the first two rows and the last two columns of the step matrix, any integer $l\in[1,2\sum\limits_{u=3}^{n-3}a_u+2n-3]$ lies in $T_{n,\lambda}^C$. Hence $E_\lambda$ is an ACM bundle.

 2.  Since $a_i=0~(2\leq i\leq n-2)$, we have
 $$ T_{n,\lambda}^D=\left(\begin{matrix}
0& 1&  2 & \dots&   n-2  &  a_1 +n-1 \\
0&0& a_{n-1}+3 &  \dots&  a_{n-1}+n-1& a_1+a_{n-1}+n \\
  \vdots&\vdots&\vdots&\ddots&\vdots&\vdots\\
 0&0&0&\dots&a_{n-1}+2n-5& a_1+a_{n-1}+2n-4\\
 0&0&0&\dots& 0& a_1+a_{n-1}+2n-3\\
 0&0&0&\dots& 0& 0\\
   \end{matrix}   \right). $$
   If $a_{n-1}\leq n-4$, then $a_{n-1}+3\leq n-1$, which means that any integer $l\in [1,a_{n-1}+2n-5]$ lies in $T^D_{n,\lambda}.$ If $a_{1}\leq n-4$, then $a_{1}+a_{n-1}+n\leq a_{n-1}+2n-4,$
   which means that any integer $l\in [1,a_1+a_{n-1}+2n-3]$ lies in $T^D_{n,\lambda}.$

\end{proof}
\bibliography{ref}

\end{document}